\documentclass{amsart}%
\usepackage{amsfonts}
\usepackage{graphicx}
\usepackage{amsmath}
\usepackage{amssymb}
\usepackage{rotating}%
\newtheorem{theorem}{Theorem}

\newtheorem{lemma}[theorem]{Lemma}

\begin{document}
\title[ ]{On The Class Of $2D$ $q$-Appell Polynomials}
\author{Marzieh Eini Keleshteri and Naz{i}m I. Mahmudov}
\address{Mathematics Department, Eastern Mediterranean University, Famagusta, North
Cyprus, via Mersin 10, Turkey}
\email{marzieh.eini@emu.edu.tr, nazim.mahmudov@emu.edu.tr}
\subjclass{Special Functions}

\begin{abstract}
In this research, as the new results of our previously proposed definition for the new class of $2D$ $q$-Appell polynomials in \cite{EiniDet}, we derive some interesting relations including the recurrence relation and partial $q$-difference equation of the aforementioned family of $q$-polynomials. Next, as some famous examples of this new defined class of $q$-polynomials, we obtain the corresponding relations to the $2D$ $q$-Bernoulli polynomials, $2D$ $q$-Euler polynomials as well as $2D$ $q$-Genocchi polynomials.
\end{abstract}
\maketitle
\section{Introduction}
In \cite{EiniDet}, Eini and Mahmudov defined $2D$ $q$-Appell Polynomials by means of the following generating function
\begin{equation}\label{2D}
A_{q}(x,y;t):=A_{q}(t)e_{q}(tx)E_{q}(ty)=\sum\limits_{n=0}^{\infty}%
A_{n,q}(x,y)\frac{t^{n}}{\left[  n\right]  _{q}!},
\end{equation}
where
\begin{equation}\label{3}
A_{q}(t):=\sum_{n=0}^{\infty}A_{n,q}\frac{t^{n}}{\left[  n\right]  _{q}!},\ \ A_{q}(t)\neq 0,
\end{equation}
is an analytic function at $t=0$, and $A_{n,q}:=A_{n,q}(0,0).$
Taking $q$-derivative of $A_{q}(x,y,t)$ with respect to the variable $x$, from one hand we obtain
\begin{align}
& D_{q,x}(A_{q}(x,y;t))=D_{q,x}(A_{q}(t)e_{q}(tx)E_{q}(ty))& \nonumber \\ \nonumber
&= tA_{q}(t)e_{q}(tx)E_{q}(ty)=\sum\limits_{n=0}^{\infty}A_{n,q}(x,y)\frac{t^{n+1}}{\left[  n\right]  _{q}!}& \\
&= \sum\limits_{n=1}^{\infty}\left[  n\right]_{q}A_{n-1,q}(x,y)\frac{t^{n}}{\left[  n\right]  _{q}!}.& \label{I1}
\end{align}
From another hand we can write
\begin{align}
& D_{q,x}(A_{q}(x,y;t))= D_{q,x}(\sum\limits_{n=0}^{\infty}A_{n,q}(x,y)\frac{t^{n}}{\left[  n\right]  _{q}!})& \nonumber \\
&= \sum\limits_{n=0}^{\infty}D_{q,x}(A_{n,q}(x,y))\frac{t^{n}}{\left[  n\right]  _{q}!}.& \label{I2}
\end{align}
Comparing the coefficients of $\frac{t^{n}}{\left[  n\right]  _{q}!}$ in the relations (\ref{I1}) and (\ref{I2}), leads to obtain
\begin{equation}\label{IDqX}
 D_{q,x}(A_{n,q}(x,y))=\left[  n\right]_{q}A_{n-1,q}(x,y).
\end{equation}
Using a similar technique for taking $q$-derivative of $A_{q}(x,y;t)$ with respect to the variable $y$, we have
\begin{equation}\label{IDqY}
 D_{q,y}(A_{n,q}(x,y))=\left[  n\right]_{q}A_{n-1,q}(x,qy).
\end{equation}
Now, according to relations (\ref{IDqX}) and (\ref{IDqY}), we define the following lowering operators
\begin{equation}\label{Ops}
\Phi_{n,{q_x}}=\frac{1}{[n]_q}D_{q,x}, \quad \Phi_{n,{q_y}}=\frac{1}{[n]_q}D_{q,y}.
\end{equation}
Therefore, we may reexpress the relations (\ref{IDqX}) and (\ref{IDqY})in the form of the following operational identities
\begin{equation}
\Phi_{n,{q_x}}A_{n,q}(x,y)= A_{n-1,q}(x,y), \text{ and} \quad  \Phi_{n,{q_y}}A_{n,q}(x,y)= A_{n-1,q}(x,qy),
\end{equation}
respectively. Eventually applying the above operators $k$ times, leads to obtain
\begin{equation}\label{OPX}
A_{n-k,q}(x,y)=\big(\Phi_{n-k,{q_x}}\circ\ldots\circ\Phi_{n,{q_x}}\big)A_{n,q}(x,y)=\frac{[n-k]_q!}{[n]_q!}D_{q,x}^{k}A_{n,q}(x,y),
\end{equation}
and
\begin{equation}\label{OPY}
A_{n-k,q}(x,q^ky)=\big(\Phi_{n-k,{q_y}}\circ\ldots\circ\Phi_{n,{q_y}}\big)A_{n,q}(x,y)=\frac{[n-k]_q!}{[n]_q!}D_{q,y}^{k}A_{n,q}(x,y),
\end{equation}
respectively.
\section{Recurrence Relations And $q$-Difference Equations For The Class Of $2D$ $q$-Appell Polynomials }
In 2002, Bretti et. al. proposed a generating function for the family of $2D$ Appell polynomials, \cite{Bretti}. They, also, obtained the corresponding recurrence relations and differential equations to the aforementioned family by calculating raising and lowering operators. In \cite{MahD}, Mahmudov applied an innovative technique in order to derive the recurrence relations and difference equations of the polynomials in the class of $q$-Appell polynomials only by using only lowering operators that are $q$-derivatives. Inspired by his novel approach, in the following we will use a similar technique in order to derive the corresponding relations to the class of $2D$ $q$-Appell polynomials.
\begin{theorem}\label{TH1}
The following linear homogeneous recurrence relation holds for the class of $2D$ $q$-Appell polynomials
\begin{equation}\label{TH1eqn1}
A_{n,q}(qx,y)=\frac{1}{\left[  n\right]_{q}}\sum_{k=1}^{n}\left[
\begin{array}{c}
n \\
k%
\end{array}
\right] _{q}q^{n-k}A_{n-k}(x,y)(\alpha_k+\frac{\beta_{k-1}}{[k]_q}y)+xq^nA_{n-1,q}(x,y), \quad n\geq1,
\end{equation}\label{TH1I1}
or equivalently,
\begin{multline}
A_{n,q}(qx,y)=q^n(x+(\alpha_1+\beta_0y)q^{-1})A_{n-1,q}(x,y)+\\
\frac{1}{\left[  n\right]_{q}}\sum_{k=1}^{n-1}\left[
\begin{array}{c}
n \\
k-1%
\end{array}
\right] _{q}q^{k-1}A_{k-1,q}(x,y)(\alpha_{n-k+1}+\frac{y}{[n-k+1]_q}\beta_{n-k}),\quad n\geq1.
\end{multline}
\end{theorem}
\begin{proof}
Starting with taking the $q$-derivative of the generating function in relation (\ref{2D}) with respect to $t$, we have
\begin{align}
&D_{q,t}(A_q(qx,y;t))=& \nonumber \\
&yA_q(t)e_q(qtx)E_q(qty)+D_{q,t}(A_q(t))e_q(qtx)E_q(qty)+qxA_q(x,y;qt).& \label{I2D1}
\end{align}
Now, multiplying both sides of the identity (\ref{I2D1}) by $t$ and factorizing $A_q(x,y;qt)$ form its left hand side, we obtain
\begin{multline}\label{I2D8}
\noindent tD_{q,t}(A_q(qx,y;t))=  \\
A_q(x,y;qt)\Big[ t\frac{D_{q,t}(A_q(t))}{A_q(qt)}+tqx+ty\frac{A_q(t)}{A_q(tq)}\Big].
\end{multline}
Suppose that $t\frac{D_{q,t}(A_q(t))}{A_q(qt)}=\sum\limits_{n=0}^{\infty}\alpha_n\frac{t^{n}}{\left[  n\right]_{q}!},$ and $t\frac{A_q(t)}{A_q(qt)}=\sum\limits_{n=0}^{\infty}\beta_n\frac{t^{n}}{\left[  n\right]_{q}!}$. Starting from taking $q$-derivative of the left hand side of relation (\ref{I2D8}) with respect to $t$ and also substituting the assumptions above in the right hand side of the same equation, we can continue as
\begin{multline}
\sum_{n=1}^{\infty}\left[  n\right]_{q}A_{n,q}(qx,y)\frac{t^n}{\left[  n\right]_{q}!}=\\
A_{q}(x,y;qt)\Big[\sum_{n=0}^{\infty}{\alpha}_{n}\frac{t^n}{\left[  n\right]_{q}!}+\sum_{n=0}^{\infty}y{\beta}_n\frac{t^{n+1}}{\left[  n\right]_{q}!}+tqx\Big].
\end{multline}
The last part of identity above can be written as
\begin{equation}
=\sum_{n=0}^{\infty}q^nA_{n,q}(x,y)\frac{t^n}{\left[  n\right]_{q}!}\Big[\sum_{n=0}^{\infty}{\alpha}_{n}\frac{t^n}{\left[  n\right]_{q}!}+y\sum_{n=1}^{\infty}[n]_q{\beta}_{n}\frac{t^{n}}{\left[  n\right]_{q}!}+tqx\Big],
\end{equation}
which is equivalent to
\begin{multline}
=\sum_{n=0}^{\infty}\sum_{k=0}^{n}\Big(\left[
\begin{array}{c}
n \\
k%
\end{array}
\right] _{q}q^{n-k}A_{n-k,q}(x,y)\alpha_k\Big)\frac{t^{n}}{\left[  n\right]_{q}!}+ \notag \\ y\sum_{n=0}^{\infty}\sum_{k=0}^{n}\Big(\left[
\begin{array}{c}
n \\
k%
\end{array}
\right] _{q}q^{n-k}A_{n-k,q}(x,y)\beta_k\Big)\frac{t^{n+1}}{\left[  n\right]_{q}!}
+x\sum_{n=0}^{\infty}q^{n+1}A_{n,q}(x,y)\frac{t^{n+1}}{\left[  n\right]_{q}!}.
\end{multline}
This means that
\begin{flalign}\label{First}
&\sum_{n=1}^{\infty}\left[  n\right]_{q}A_{n,q}(qx,y)\frac{t^n}{\left[  n\right]_{q}!}=A_{0,q}(x,y)\alpha_0+&\notag \\
&\sum_{n=1}^{\infty}\Bigg(\sum_{k=0}^{n}\left[
\begin{array}{c}
n \\
k%
\end{array}
\right] _{q}q^{n-k}A_{n-k,q}(x,y)\alpha_k+y\sum_{k=0}^{n-1}\left[
\begin{array}{c}
n-1 \\
k%
\end{array}
\right] _{q}[n]_qq^{n-k-1}A_{n-k-1,q}(x,y)\beta_k+& \notag\\
&x[n]_q^nA_{n-1,q}(x,y)\Bigg)\frac{t^{n}}{\left[  n\right]_{q}!}.&
\end{flalign}
\begin{equation}
= \sum_{n=0}^{\infty}\sum_{k=0}^{n}\Big(\left[
\begin{array}{c}
n \\
k%
\end{array}
\right] _{q}q^{n-k}A_{n-k}(x,y)(\alpha_k+\beta_ky)+[n]_qxq^nA_{n-1,q}(x,y)\Big)\frac{t^{n}}{\left[  n\right]_{q}!}.
\end{equation}
Comparing the coefficients of $\frac{t^{n}}{\left[  n\right]_{q}!}$ in both sides of relation (\ref{First}) and noting to the fact that $\alpha_0=0$, lead to obtain the following identity for $n\geq1$
\begin{multline}
[n]_q A_{n,q}(qx,y)=
\sum_{k=1}^{n}\left[
\begin{array}{c}
n \\
k%
\end{array}
\right] _{q}q^{n-k}A_{n-k,q}(x,y)\alpha_k+ \notag \\ y\sum_{k=1}^{n}\left[
\begin{array}{c}
n-1 \\
k-1%
\end{array}
\right] _{q}[n]_qq^{n-k}A_{n-k,q}(x,y)\beta_{k-1}+x[n]_qq^{n}A_{n-1,q}(x,y),
\end{multline}
whence the result.
\end{proof}
\begin{theorem}\label{TH2}
The following partial $q$-difference equations hold for the polynomials in the class of $2D$ $q$-Appell
\begin{equation}
 \Bigg[\sum_{k=1}^{n}\frac{q^{n-k}}{[k]_q!}(\alpha_k+\frac{\beta_{k-1}}{[k]_q}y)D_{q,x}^{k}+xq^nD_{q,x}\Bigg]A_{n,q}(x,y)-[n]_qA_{n,q}(qx,y)=0.
\end{equation}
\begin{equation}
\sum_{k=1}^{n}\frac{q^{n-k}}{[k]_q!}(\alpha_k+\frac{\beta_{k-1}}{q^{k}[k]_q}y)D_{q,y}^{k}A_{n,q}(x,\frac{y}{q^k})+xq^nD_{q,y}A_{n,q}(x,\frac{y}{q})-[n]_qA_{n,q}(qx,y)=0.
\end{equation}
\end{theorem}
\begin{proof}
The proof is the direct result of replacing relations (\ref{OPX}) and (\ref{OPY}) in the linear homogeneous recurrence relation (\ref{TH1eqn1}) given in Theorem (\ref{TH1}), respectively.
\end{proof}
\section{$q$-Difference Equations For Various Members Of The Family Of $q$-Appell Polynomials}
Choosing different functions as $A_{q}(t)$ in Definition (\ref{2D}), leads to generate different members of $2D$ $q$-Appell polynomials. In the following we introduce some of the most famous $2D$ $q$-Appell polynomials and derive the corresponding recurrence relations and partial $q$-difference equations to them.
\subsection{$2D$ $q$-Bernoulli polynomials}
Taking $A_{q}(t)=\frac{t}{e_{q}(t)-1}$ in Definition (\ref{2D}), leads to obtain $2D$ $q$-Bernoulli polynomials, $B_{n,q}(x,y)$, \cite{calitz1}, \cite{Mah}.
\begin{equation}
B_{q}(x,y;t):=\frac{t}{e_{q}(t)-1}e_{q}(tx)E_{q}(ty)=\sum\limits_{n=0}^{\infty}%
B_{n,q}(x,y)\frac{t^{n}}{\left[  n\right]  _{q}!},
\end{equation}
\begin{lemma}
Suppose that
\begin{equation}
t\frac{D_{q,t}(A_q(t))}{A_q(qt)}=t\frac{D_{q,t}{\frac{t}{e_{q}(t)-1}}}{\frac{qt}{e_{q}(qt)-1}}=\sum\limits_{n=0}^{\infty}\alpha_n\frac{t^{n}}{\left[  n\right]  _{q}!},
\end{equation}
and
\begin{equation}
\frac{A_q(t)}{A_q(qt)}=\frac{\frac{t}{e_{q}(t)-1}}{\frac{qt}{e_{q}(qt)-1}}=\sum\limits_{n=0}^{\infty}\beta_n\frac{t^{n}}{\left[  n\right]  _{q}!},
\end{equation}
then
\begin{equation}
\alpha_n=\frac{-1}{q}b_{n,q},\quad \alpha_1=\frac{-1}{[2]_q},
\end{equation}
and
\begin{equation}
\beta_n=\frac{q-1}{q}\sum_{k=0}^{n}\left[
\begin{array}{c}
n \\
k%
\end{array}
\right] _{q}b_{k,q}, \quad \text{for }n\geq1 \quad \text{and }\beta_0=1,
\end{equation}
where $b_{n,q}=B_{n,q}(0,0)$ is the n-th $q$-Bernoulli number and can be obtained from the generating function $\frac{t}{e_{q}(t)-1}=\sum_{n=0}^{\infty}b_{n,q}\frac{t^n}{[n]_q!}.$
\end{lemma}
\begin{theorem}
The following linear homogeneous recurrence relation holds for the class of $2D$ $q$-Bernoulli polynomials for every $n\geq1$
\begin{flalign}
&B_{n,q}(qx,y)=q^n(x+(\frac{-1}{[2]_q}+y)q^{-1})B_{n-1,q}(x,y)+ \notag&\\
&\frac{1}{\left[  n\right]_{q}}\sum_{k=1}^{n-1}\left[
\begin{array}{c}
n \\
k-1%
\end{array}
\right] _{q}q^{k-2}B_{k-1,q}(x,y)(-b_{n-k+1,q}+\frac{(q-1)y}{[n-k+1]_q}\sum_{l=0}^{n-k}\left[
\begin{array}{c}
n-k \\
l%
\end{array}
\right] _{q}b_{l,q}).&
\end{flalign}
\end{theorem}
\begin{theorem}
The following partial $q$-difference equations hold for the polynomials in the class of $2D$ $q$-Bernoulli
\begin{flalign}
&\Bigg[(xq^n+y-\frac{1}{[2]_q})D_{q,x}+
\sum_{k=2}^{n}\frac{q^{n-k-1}}{[k]_q!}(-b_{k,q}+(q-1)y\sum_{l=0}^{k-1}\left[
\begin{array}{c}
k-1 \\
l%
\end{array}
\right] _{q}b_{l,q})D_{q,x}^{k}\Bigg]\notag& \\ &\times B_{n,q}(x,y)-[n]_qB_{n,q}(qx,y)=0.&
\end{flalign}
\begin{flalign}
&(xq^n+\frac{y}{q}-\frac{1}{[2]_q})D_{q,y}A_{n,q}(x,\frac{y}{q})+\sum_{k=2}^{n}\frac{q^{n-k-1}}{[k]_q!}(-b_{k,q}+\frac{(q-1)y}{q^{k}[k]_q}\sum_{l=0}^{k}\left[
\begin{array}{c}
k \\
l%
\end{array}
\right] _{q}b_{l,q})\notag & \\ & \times D_{q,y}^{k}B_{n,q}(x,\frac{y}{q^k})-[n]_qA_{n,q}(qx,y)=0.&
\end{flalign}
\end{theorem}
\subsection{$2D$ $q$-Euler polynomials}
Taking $A_{q}(t)=\frac{2}{e_{q}(t)+1}$ in Definition (\ref{2D}), leads to obtain $2D$ $q$-Euler polynomials, $E_{n,q}(x,y)$, as follows\cite{calitz1}, \cite{Mah}.
\begin{equation}
E_{q}(x,y;t):=\frac{2}{e_{q}(t)+1}e_{q}(tx)E_{q}(ty)=\sum\limits_{n=0}^{\infty}%
E_{n,q}(x,y)\frac{t^{n}}{\left[  n\right]  _{q}!},
\end{equation}
\begin{lemma}
Suppose that
\begin{equation}
t\frac{D_{q,t}(A_q(t))}{A_q(qt)}=t\frac{D_{q,t}{\frac{2}{e_{q}(t)+1}}}{\frac{2}{e_{q}(qt)+1}}=\sum\limits_{n=0}^{\infty}\alpha_n\frac{t^{n}}{\left[  n\right]  _{q}!},
\end{equation}
and
\begin{equation}
\frac{A_q(t)}{A_q(qt)}=\frac{\frac{2}{e_{q}(t)+1}}{\frac{2}{e_{q}(qt)+1}}=\sum\limits_{n=0}^{\infty}\beta_n\frac{t^{n}}{\left[  n\right]  _{q}!},
\end{equation}
then
\begin{equation}
\alpha_n=\frac{1}{2}E_{n-1,q},\quad \alpha_1=\frac{-1}{2},
\end{equation}
and
\begin{equation}
\beta_n=\frac{q-1}{2}\sum_{k=0}^{n}\left[
\begin{array}{c}
n \\
k%
\end{array}
\right] _{q}E_{k,q}, \quad \text{for }n\geq1 \quad \text{and }\beta_0=\frac{q+1}{2},
\end{equation}
where $E_{n,q}=E_{n,q}(0,0)$ is the n-th $q$-Euler number and can be obtained from the generating function $\frac{2}{e_{q}(t)+1}=\sum_{n=0}^{\infty}E_{n,q}\frac{t^n}{[n]_q!}.$
\end{lemma}
\begin{theorem}
The following linear homogeneous recurrence relation holds for the class of $2D$ $q$-Euler polynomials for every $\quad n\geq1$
\begin{multline*}
A_{n,q}(qx,y)=q^n(x+\frac{(q+1)y-1}{2q})E_{n-1,q}(x,y)+\\ 
\frac{1}{\left[  n\right]_{q}}\sum_{k=1}^{n-1}\left[
\begin{array}{c}
n \\
k-1%
\end{array}
\right] _{q}q^{k-1}E_{k-1,q}(x,y)(\frac{1}{2}E_{n-k,q}+\frac{(q-1)y}{q[n-k+1]_q}\sum_{l=0}^{n-k}\left[
\begin{array}{c}
n-k \\
l%
\end{array}
\right] _{q}E_{l,q}).
\end{multline*}
\end{theorem}
\begin{theorem}
The following partial $q$-difference equations hold for the polynomials in the class of $2D$ $q$-Euler
\begin{flalign}
 &\Bigg[(xq^n+\frac{(q+1)y}{2[k]_q}-\frac{1}{2})D_{q,x}+\sum_{k=2}^{n}\frac{q^{n-k}}{[k]_q!}(\frac{1}{2}E_{k-1,q}+\frac{(q-1)y}{2[k]_q}\sum_{l=0}^{k-1}\left[
\begin{array}{c}
k-1 \\
l%
\end{array}
\right] _{q}E_{l,q})D_{q,x}^{k}\Bigg]\notag& \\ \nonumber &\times E_{n,q}(x,y)-[n]_qE_{n,q}(qx,y)=0.&
\end{flalign}
\begin{flalign}
&(xq^n+\frac{(q+1)y}{2q[k]_q}-\frac{1}{2})D_{q,y}E_{n,q}(x,\frac{y}{q})+\sum_{k=2}^{n}\frac{q^{n-k}}{[k]_q!}(\frac{1}{2}E_{k-1,q}+\frac{(q-1)y}{2q^k[k]_q}\sum_{l=0}^{k-1}\left[
\begin{array}{c}
k-1 \\
l%
\end{array}
\right] _{q}E_{l,q})\notag & \\ \nonumber & \times D_{q,y}^{k}E_{n,q}(x,\frac{y}{q^k})-[n]_qE_{n,q}(qx,y)=0.&
\end{flalign}
\end{theorem}
\subsection{$2D$ $q$-Genocchi polynomials}
Taking $A_{q}(t)=\frac{2t}{e_{q}(t)+1}$ in Definition (\ref{2D}), leads to obtain $2D$ $q$-Genocchi polynomials, $G_{n,q}(x,y)$, as follows
\begin{equation}
G_{q}(x,y;t):=\frac{2t}{e_{q}(t)+1}e_{q}(tx)G_{q}(ty)=\sum\limits_{n=0}^{\infty}%
G_{n,q}(x,y)\frac{t^{n}}{\left[  n\right]  _{q}!},
\end{equation}
\begin{lemma}
Suppose that
\begin{equation}
t\frac{D_{q,t}(A_q(t))}{A_q(qt)}=t\frac{D_{q,t}{\frac{2t}{e_{q}(t)+1}}}{\frac{2tq}{e_{q}(qt)+1}}=\sum\limits_{n=0}^{\infty}\alpha_n\frac{t^{n}}{\left[  n\right]  _{q}!},
\end{equation}
and
\begin{equation}
\frac{A_q(t)}{A_q(qt)}=\frac{\frac{2t}{e_{q}(t)+1}}{\frac{2tq}{e_{q}(qt)+1}}=\sum\limits_{n=0}^{\infty}\beta_n\frac{t^{n}}{\left[  n\right]  _{q}!},
\end{equation}
then
\begin{equation}
\alpha_n=\frac{1}{2q}G_{n,q},\quad \text{for } n\geq 2\text{, and } \alpha_0=\frac{1}{q}\text{, }\alpha_1=\frac{-1}{q},
\end{equation}
and
\begin{equation}
\beta_n=\frac{q-1}{2q}\sum_{k=0}^{n}\left[
\begin{array}{c}
n \\
k%
\end{array}
\right] _{q}G_{k,q}, \quad \text{for }n\geq1 \quad \text{and }\beta_0=\frac{1}{q},
\end{equation}
where $G_{n,q}=G_{n,q}(0,0)$ is the n-th $q$-Genocchi number and can be obtained from the generating function $\frac{2t}{e_{q}(t)+1}=\sum_{n=0}^{\infty}G_{n,q}\frac{t^n}{[n]_q!}.$
\end{lemma}
\begin{theorem}
The following linear homogeneous recurrence relation holds for the class of $2D$ $q$-Genocchi polynomials for every $n\geq1$
\begin{multline}
G_{n,q}(qx,y)=q^n(x+\frac{y-1}{q^2})G_{n-1,q}(x,y)+\\
\frac{1}{2\left[  n\right]_{q}}\sum_{k=1}^{n-1}\left[
\begin{array}{c}
n \\
k-1%
\end{array}
\right] _{q}q^{k-2}G_{k-1,q}(x,y)(G_{n-k+1,q}+\frac{(q-1)y}{[n-k+1]_q}\sum_{l=0}^{n-k}\left[
\begin{array}{c}
n-k \\
l%
\end{array}
\right] _{q}G_{l,q}).
\end{multline}
\end{theorem}
\begin{theorem}
The following partial $q$-difference equations hold for the polynomials in the class of $2D$ $q$-Genocchi
\begin{flalign}
& \Bigg[(xq^n+\frac{y-1}{q})D_{q,x}+\sum_{k=2}^{n}\frac{q^{n-k-1}}{2[k]_q!}(G_{k,q}+\frac{(q-1)y}{[k]_q}\sum_{l=0}^{k-1}\left[
\begin{array}{c}
k-1 \\
l%
\end{array}
\right] _{q}G_{l,q})D_{q,x}^{k}\Bigg]\notag & \\ &  \times G_{n,q}(x,y)-[n]_qG_{n,q}(qx,y)=0.&
\end{flalign}
\begin{flalign}
&(xq^n+\frac{y-q}{q^2})D_{q,y}G_{n,q}(x,\frac{y}{q})+\sum_{k=2}^{n}\frac{q^{n-k-1}}{2[k]_q!}(G_{k,q}+\frac{(q-1)y}{q^k[k]_q}\sum_{l=0}^{k-1}\left[
\begin{array}{c}
k-1 \\
l%
\end{array}
\right] _{q}G_{l,q}) \notag & \\ &  \times D_{q,y}^{k}G_{n,q}(x,\frac{y}{q^k})-[n]_qG_{n,q}(qx,y)=0.&
\end{flalign}
\end{theorem}

\end{document}